\documentclass[a4paper]{amsart}

\usepackage[english]{babel}
\usepackage{amsmath}
\usepackage{amsfonts}
\usepackage{amsthm}
\usepackage{dsfont}

\theoremstyle{plain}
\newtheorem{theorem}{Theorem}[section]

\newtheorem{corollary}[theorem]{Corollary}
\theoremstyle{definition}

\theoremstyle{remark}
\newtheorem{remark}[theorem]{Remark}

\numberwithin{equation}{section}

\newcommand{\eps}{\ensuremath{\varepsilon}}
\newcommand{\esp}{\ensuremath{\mathbb{E}}}
\newcommand{\ind}{\ensuremath{\mathds{1}}}
\newcommand{\N}{\ensuremath{\mathbb{N}}}
\newcommand{\R}{\ensuremath{\mathbb{R}}}

\begin{document}

\title[Weak convergence towards two independent Gaussian processes]{Weak convergence towards two independent Gaussian processes from a unique Poisson process}
\author[X. Bardina and D. Bascompte]{Xavier Bardina$^{*}$ and David Bascompte}
\address{Departament de Matem\`atiques\\
Universitat Aut\`onoma de Barcelona\\ 08193 Bellaterra (Spain)}
\email{bardina@mat.uab.cat, bascompte@mat.uab.cat}
\thanks{$^{*}$ Corresponding author\\ The authors are partially supported by MEC-Feder Grant MTM2006-06427}
\begin{abstract}
We consider two independent Gaussian processes that admit a
representation in terms of a stochastic integral of a deterministic
kernel with respect to a standard Wiener process. In this paper we
construct two families of processes, from a unique Poisson process,
the finite dimensional distributions of which converge in law
towards the finite dimensional distributions of the two independent
Gaussian processes.

As an application of this result we obtain families of processes
that converge in law towards fractional Brownian motion and
sub-fractional Brownian motion.
\end{abstract} \maketitle


\section{Introduction and preliminaries}

Let $f(t,\cdot)$ and $g(t,\cdot)$ be functions of
$\mathrm{L^2(\R^+)}$ for all $t\in[0,T]$, \mbox{$T>0$} and consider
the processes given by
\begin{equation}\label{nova1}Y^f=\left\{\int_0^\infty f(t,s)\mathrm{d}W_s,
t\in[0,T]\right\}\end{equation} and
\begin{equation}\label{nova2}\tilde{Y}^g=\left\{\int_0^\infty
g(t,s)\mathrm{d}\tilde{W}_s,t\in[0,T]\right\}\end{equation} where
$W=\{W_s,s\geq0\}$ and $\tilde{W}=\{\tilde{W}_s,s\geq0\}$ are
independent standard Brownian motions.

The aim of this paper is to construct two families of processes,
from a unique Poisson process, that converge, in the sense of the
finite dimensional distributions, to the processes $Y^f$ and $\tilde
Y^g$. We will use this result later in order to prove weak
convergence results towards different kinds of processes such as
fractional Brownian motion and sub-fractional Brownian motion.

It is well known the result by Stroock (see \cite{Stroock-1982})
where it is shown that the family of processes
\[
\left\{x_\eps(t)=\frac1\eps\int_0^t(-1)^{N_\frac{s}{\eps^2}}\,\mbox{d}s,\quad
t\in[0,T]\right\},
\]
defined from the kernels
$\theta_\eps=\frac1\eps(-1)^{N_\frac{s}{\eps^2}}$, converges in law
in $\mathcal{C}([0,T])$ to a standard Brownian motion, where
$N=\{N_s,s\geq 0\}$ is a standard Poisson process. This kind of
processes were introduced by Kac in \cite{Kac-1974} in order to
write the solution of telegrapher's equation in terms of Poisson
process.

On the other hand, Delgado and Jolis (see \cite{Delgado-2000})
extend this result to processes represented by a stochastic
integral, with respect to a standard Wiener process, of a
deterministic kernel that satisfies some regularity conditions.

A generalization of Stroock's result can be found in
\cite{Bardina-2001}, where it is proved that the family
\begin{equation}\label{e:complex}
\left\{x_\eps^\theta(t)=\frac{2}{\eps}\int_0^te^{i\theta
N_{\frac{2s}{\eps^2}}}\,\mbox{d}s,\quad t\in[0,T]\right\}
\end{equation}
converges in law in $\mathcal{C}([0,T])$ to a complex Brownian
motion, for $\theta\in(0,\pi)\cup(\pi,2\pi)$. Particularly, the real
part and the imaginary part of (\ref{e:complex}) tend to independent
standard Brownian motions.

In this paper, given \mbox{$\{N_s,s\geq0\}$} a standard Poisson
process and $\theta\in(0,\pi)\cup(\pi,2\pi)$, we consider the
following families of approximating processes
\begin{equation}\label{e:cos_f}
Y_{\eps}^f=\left\{\frac{2}{\eps}\int_0^\infty f(t,s)\cos\left(\theta
N_{\frac{2s}{\eps^2}}\right)\mbox{d}s,\quad t\in[0,T]\right\}
\end{equation}
and
\begin{equation}\label{e:sin_f}
\tilde{Y}_{\eps}^g=\left\{\frac{2}{\eps}\int_0^\infty
g(t,s)\sin\left(\theta N_{\frac{2s}{\eps^2}}\right)\mbox{d}s,\quad
t\in[0,T]\right\}.
\end{equation}

The main result of this paper is the proof that the finite
dimensional distributions of the processes $Y_{\eps}^f$ and
$\tilde{Y}_{\eps}^g$ converge in law to the finite dimensional
distributions of the processes $Y^f$ and $\tilde{Y}^g$ given by
(\ref{nova1}) and (\ref{nova2}), respectively.

It is important to note that the processes $Y_{\eps}^f$ and
$\tilde{Y}_{\eps}^g$ are both functionally dependent. Nevertheless,
integrating and taking limits, we obtain two independent processes.

As an application of this result it can be obtained approximations
for different examples of centered Gaussian processes, among others,
fractional Brownian motion and sub-fractional Brownian motion.

Recall that \emph{fractional Brownian motion} (fBm for short)
$B^H=\{B^H(t),t\geq0\}$ is a centered Gaussian process with
covariance function
\begin{equation}\label{e:frac_cov}
\mbox{Cov}(B^H_t,B^H_s)=\frac{1}{2}\left(s^{H}+t^{H}-|s-t|^{H}\right)
\end{equation}
where $H\in (0,2)$. Usually fBm is defined with Hurst parameter
belonging to the interval $(0,1)$ with the corresponding covariance,
but in order to compare it with sub-fBm we use the stated
representation with $H\in(0,2)$.

On the other hand, \emph{Sub-fractional Brownian motion} (sub-fBm
for brevity) $S^H=\{S^H(t),t\geq0\}$ is a centered Gaussian process
with covariance function
\begin{equation}\label{e:subfrac_cov}
\mbox{Cov}(S^H_t,S^H_s)=s^{H}+t^{H}-\frac{1}{2}\left[(s+t)^{H}+|s-t|^{H}\right]
\end{equation}
where $H\in (0,2)$.

This process was introduced by Bojdecki {\it et al.} in 2004 (see
\cite{Boj1}) as an intermediate process between standard Brownian
motion and fractional Brownian motion.  Note that both fBm and
sub-fBm are standard Brownian motions for $H=1$.

For $H\neq1$, sub-fBm preserves some of the main properties of fBm, such as long-range dependence, but its increments are not stationary;
they are more weakly correlated on non-overlapping intervals than fBm ones, and their covariance decays polynomially at a higher rate as the
distance between the intervals tends to infinity. For a more detailed discussion of sub-fBm and its properties we refer the reader to \cite{Boj1}.
Some properties of this process have also been studied in \cite{Tudor1} and \cite{Tudor2}. On the other hand there is an extension of sub-fBm
in \cite{Boj2}.

In  \cite{Tudor3} (see Theorem \ref{t:descomposicio} below) the
authors obtain a decomposition of the sub-fBm in terms of fBm and
another process with absolutely continuous trajectories,
$X^H=\{X_t^H,t\geq0\}$, which is defined by Lei and Nualart in
\cite{Lei-2008} by
\begin{equation}\label{e:david}
X_t^H=\int_0^{\infty}(1-e^{-r t})r^{-\frac{1+H}{2}}\,\mathrm{d}W_r
\end{equation}
where $W$ is a standard Brownian motion. Lei and Nualart introduce this process in order to obtain a decomposition of bifractional Brownian motion
 into the sum of a transformation of $X_t^H$ and a fBm.

The decomposition is different for $H\in(0,1)$ and $H\in(1,2)$. In
the first case, sub-fBm is obtained as a sum of two independent
processes, a fBm and the process defined by (\ref{e:david}), while
for $H\in(1,2)$ is fBm that is decomposed into the sum of the
process (\ref{e:david}) and a sub-fBm, being these independents.

The paper is organized as follows.  In Section 2 we will prove the
general result of weak convergence, in the sense of the finite
dimensional distributions, towards integrals of functions of
$L^2(\R^{+})$ with respect to two independent standard Brownian
motions. This theorem permits us to obtain, in Section 3, results of
convergence in law, in the space $\mathcal C([0,T])$, towards fBm,
the process defined in (\ref{e:david}) and, finally, sub-fBm with
parameter $H\in(0,1)$ using the decomposition of this process as a
sum of two independent processes.

Positive constants, denoted by $C$, with possible subscripts indicating appropriate parameters, may vary from line to line.


\section{General convergence result}

In this section we prove the main result of weak convergence in the
sense of the finite dimensional distributions. We will use this
result later in order to prove weak convergence results towards
fractional Brownian motion and sub-fractional Brownian motion.


\begin{theorem}\label{t:conv}
Let $f(t,\cdot)$ and $g(t,\cdot)$ be functions of
$\mathrm{L^2(\R^+)}$ for all $t\in[0,T]$, \mbox{$T>0$}, let
\mbox{$\{N_s,s\geq0\}$} be a standard Poisson process and
$\theta\in(0,\pi)\cup(\pi,2\pi)$. Define the processes $Y^f$ and
$\tilde Y^g$, which are given by $Y^f=\{\int_0^\infty
f(t,s)\mathrm{d}W_s, t\in[0,T]\}$ and $\tilde{Y}^g=\{\int_0^\infty
g(t,s)\mathrm{d}\tilde{W}_s,t\in[0,T]\}$ and where
$W=\{W_s,s\geq0\}$ and $\tilde{W}=\{\tilde{W}_s,s\geq0\}$ are
independent standard Brownian motions. We also define the following
processes
\begin{equation}\label{e:cos_f}
Y_{\eps}^f=\left\{\frac{2}{\eps}\int_0^\infty f(t,s)\cos\left(\theta
N_{\frac{2s}{\eps^2}}\right)\mbox{d}s,\quad t\in[0,T]\right\}
\end{equation}
and
\begin{equation}\label{e:sin_f}
\tilde{Y}_{\eps}^g=\left\{\frac{2}{\eps}\int_0^\infty
g(t,s)\sin\left(\theta N_{\frac{2s}{\eps^2}}\right)\mbox{d}s,\quad
t\in[0,T]\right\}.
\end{equation}

Then, the finite dimensional distributions of the processes
$\{Y_\eps^f\}$ and $\{\tilde{Y}_\eps^g\}$ converge in law to the
finite dimensional distributions of the processes $Y^f$ and
$\tilde{Y}^g$.
\end{theorem}

\begin{proof}

Taking into account that the proof is valid for any fixed
$t\in[0,T]$, by abuse of notation we will write $f(s)$ instead of
$f(t,s)$. Slightly modifying the proof of Theorem 1 in
\cite{Delgado-2000}, in order to prove the weak convergence, in the
sense of the finite dimensional distributions, it suffices to show
that

\begin{equation}\label{e:m2}
\esp\left[(Y_{\eps}^f)^2\right]\leq C\left(\int_0^\infty
f^2(s)\,\mbox{d}s\right)\!,\quad
\esp\left[(\tilde{Y}_{\eps}^g)^2\right]\leq C\left(\int_0^\infty
g^2(s)\,\mbox{d}s\right).
\end{equation}
Observe that defining
\[Z_\eps^f=Y_\eps^f+i\tilde{Y}_\eps^f=\frac{2}{\eps}\int_0^\infty f(s)e^{i\theta N_{\frac{2s}{\eps^2}}}\mbox{d}s\]
we have
$\esp[Z_\eps^f\bar{Z}_\eps^f]=\esp[(Y_\eps^f)^2+(\tilde{Y}_\eps^f)^2]$.
Therefore if we prove $\esp[Z_\eps^f\bar{Z}_\eps^f]\leq C\lVert
f\rVert_2^2$, where $\lVert\cdot\rVert_2$ is the $L^2(\R^+)$ norm,
the stated convergence follows.

\begin{align*}
\esp[Z_\eps^f\bar{Z}_\eps^f]&=\esp\left[\frac{2}{\eps}\int_0^\infty f(s)e^{i\theta N_{\frac{2s}{\eps^2}}}\mbox{d}s\frac{2}{\eps}\int_0^\infty f(r)e^{-i\theta N_{\frac{2r}{\eps^2}}}\mbox{d}r\right] \\
                       &=\frac{4}{\eps^2}\esp\left[\int_0^\infty \int_0^\infty f(s)f(r)e^{i\theta\left(N_{\frac{2s}{\eps^2}}-N_{\frac{2r}{\eps^2}}\right)}\mbox{d}s\,\mbox{d}r\right]\\
                       &=\frac{4}{\eps^2}\int_0^\infty \int_0^\infty \ind_{\{r\leq s\}}f(s)f(r)\esp\left[e^{i\theta\left(N_{\frac{2s}{\eps^2}}-N_{\frac{2r}{\eps^2}}\right)}\right]\mbox{d}r\,\mbox{d}s\\
                       &\quad+\frac{4}{\eps^2}\int_0^\infty \int_0^\infty \ind_{\{s\leq r\}}f(s)f(r)\esp\left[e^{-i\theta\left(N_{\frac{2r}{\eps^2}}-N_{\frac{2s}{\eps^2}}\right)}\right]\mbox{d}s\,\mbox{d}r.
\end{align*}
Since $\esp[e^{i\theta X}]\!=\!e^{-\lambda(1-e^{i\theta})}$ and
$\esp[e^{-i\theta X}]\!=\!e^{-\lambda(1-e^{-i\theta})}$, being $X$ a
Poisson random variable of parameter $\lambda$, we obtain

\begin{align*}
\esp[Z_\eps^f\bar{Z}_\eps^f]&=\frac{4}{\eps^2}\int_0^\infty \int_0^\infty \ind_{\{r\leq s\}}f(s)f(r)e^{-2\frac{s-r}{\eps^2}(1-e^{i\theta})}\mbox{d}r\,\mbox{d}s\\
                        &\quad+\frac{4}{\eps^2}\int_0^\infty \int_0^\infty \ind_{\{s\leq r\}}f(s)f(r)e^{-2\frac{r-s}{\eps^2}(1-e^{-i\theta})}\mbox{d}r\,\mbox{d}s\\
                        &\leq\frac{4}{\eps^2}\int_0^\infty \int_0^\infty \ind_{\{r\leq s\}}\lvert f(s)f(r)\rvert\,e^{-2\frac{s-r}{\eps^2}(1-\cos\theta)}\mbox{d}r\,\mbox{d}s\\
                        &\quad+\frac{4}{\eps^2}\int_0^\infty \int_0^\infty \ind_{\{s\leq r\}}\lvert f(s)f(r)\rvert\,e^{-2\frac{r-s}{\eps^2}(1-\cos\theta)}\mbox{d}r\,\mbox{d}s.
\end{align*}
Using the inequality $\lvert
f(s)f(r)\rvert\leq\frac{1}{2}\left(f^2(s)+f^2(r)\right)$ and noting
that, by means of a change of variables, the last two integrals are
the same we have that
\begin{align*}
\esp[Z_\eps^f\bar{Z}_\eps^f]&\leq\frac{4}{\eps^2}\int_0^\infty \int_0^\infty \ind_{\{s\leq r\}}\left(f^2(s)+f^2(r)\right)e^{-2\frac{r-s}{\eps^2}(1-\cos\theta)}\mbox{d}r\,\mbox{d}s\\
                        &=\frac{4}{\eps^2}\!\left(\!\int_0^\infty \!\!\!f^2(s)\!\int_s^\infty \!\!\!e^{-2\frac{r-s}{\eps^2}(1-\cos\theta)}\mbox{d}r\,\mbox{d}s+\!\!\int_0^\infty \!\!\!f^2(r)\!\int_0^r\!\!\!e^{-2\frac{r-s}{\eps^2}(1-\cos\theta)}\mbox{d}s\,\mbox{d}r\!\right)\\
                        &=2\left(\int_0^\infty f^2(s)\left(\frac{1}{1-\cos\theta}\right)\mbox{d}s+\int_0^\infty f^2(r)\left(\frac{1-e^{-2\frac{r}{\eps^2}(1-\cos\theta)}}{1-\cos\theta}\right)\mbox{d}r\right)\\
                        &\leq\frac{4}{1-\cos\theta}\int_0^\infty f^2(s)\,\mathrm{d}s.
\end{align*}

Then, the convergence of the finite dimensional distributions has
been proved and it remains to prove the independence of the limit
processes. We begin by proving that the family
$\{Y_\eps^f\tilde{Y}_\eps^g\}_{\eps>0}$ is uniformly integrable.
Indeed, we will prove that
$\sup_{\eps>0}\esp\left[(Y_\eps^f\tilde{Y}_\eps^g)^2\right]<\infty$.
Using H\"older's inequality we have
\[
\sup_{\eps>0}\esp\left[(Y_\eps^f\tilde{Y}_\eps^g)^2\right]\leq\sup_{\eps>0}\left(\esp[(Y_\eps^f)^4]\right)^\frac12\left(\esp[(\tilde{Y}_\eps^g)^4]\right)^\frac12.
\]
In order to prove that the last expression is finite, we will show
that
\begin{equation}\label{e:m4}
\esp\left[(Y_{\eps}^f)^4\right]\leq C\left(\int_0^\infty
f^2(s)\,\mbox{d}s\right)^2\!,\quad
\esp\left[(\tilde{Y}_{\eps}^g)^4\right]\leq C\left(\int_0^\infty
g^2(s)\,\mbox{d}s\right)^2.
\end{equation}
Being $Z_\eps^f$ like before, we can prove (\ref{e:m4}) showing that
$\esp[(Z_\eps^f\bar{Z}_\eps^f)^2]\leq C\lVert f\rVert_2^4$.

\begin{align*}
\esp[(Z_\eps^f\bar{Z}_\eps^f)^2]&=\frac{16}{\eps^4}\esp\left[\int_{[0,\infty)^4}f(s_1)\dotsm f(s_4)e^{i\theta\left(N_{\frac{2s_1}{\eps^2}}+N_{\frac{2s_2}{\eps^2}}-N_{\frac{2s_3}{\eps^2}}-N_{\frac{2s_4}{\eps^2}}\right)}\mbox{d}s_1\dotsm\mbox{d}s_4\right]\\
                        &=\frac{64}{\eps^4}\int_{[0,\infty)^4}\ind_{\{s_1\leq\dotsb\leq s_4\}}f(s_1)\dotsm f(s_4)\esp\left[E_1+\dotsb+E_6\right]\mbox{d}s_1\dotsm\mbox{d}s_4
\end{align*}
where
\begin{align*}
E_1\!&=\!e^{i\theta\Big(N_{\frac{2s_1}{\eps^2}}+N_{\frac{2s_2}{\eps^2}}-N_{\frac{2s_3}{\eps^2}}-N_{\frac{2s_4}{\eps^2}}\Big)}\!\!=\!e^{-i\theta\Big(N_{\frac{2s_4}{\eps^2}}-N_{\frac{2s_3}{\eps^2}}+2\big(N_{\frac{2s_3}{\eps^2}}-N_{\frac{2s_2}{\eps^2}}\big)+N_{\frac{2s_2}{\eps^2}}-N_{\frac{2s_1}{\eps^2}}\Big)},\\
E_2\!&=e^{-i\theta\left(N_{\frac{2s_4}{\eps^2}}-N_{\frac{2s_3}{\eps^2}}+N_{\frac{2s_2}{\eps^2}}-N_{\frac{2s_1}{\eps^2}}\right)},
E_3\!=e^{i\theta\left(N_{\frac{2s_4}{\eps^2}}-N_{\frac{2s_3}{\eps^2}}-\big(N_{\frac{2s_2}{\eps^2}}-N_{\frac{2s_1}{\eps^2}}\big)\right)},
\end{align*}
$E_4=\overline{E_3}$, $E_5=\overline{E}_2$, $E_6=\overline{E}_1$. To
obtain the last expression note that we can arrange
$s_1,s_2,s_3,s_4$ in 24 different ways and due to the symmetry
between $s_1$ and $s_2$ and between $s_3$ and $s_4$ we have 6
possible different situations, $E_1,\dots, E_6$, each one repeated 4
times. By means of the properties of Poisson process we have
\[\lVert\esp[E_1]\rVert,\lVert\esp[E_2]\rVert,\lVert\esp[E_3]\rVert\leq e^{-2\frac{s_4-s_3}{\eps^2}(1-\cos\theta)}e^{-2\frac{s_2-s_1}{\eps^2}(1-\cos\theta)}\]
and we can conclude
\begin{align*}
\esp[(Z_\eps^f\bar{Z}_\eps^f)^2]&\leq\frac{384}{\eps^4}\int_{[0,\infty)^4}\ind_{\{s_1\leq\dotsb\leq s_4\}}\lvert f(s_1)\dotsm f(s_4)\rvert \\
    &\qquad\qquad\qquad e^{-2\frac{s_4-s_3}{\eps^2}(1-\cos\theta)}e^{-2\frac{s_2-s_1}{\eps^2}(1-\cos\theta)}\mbox{d}s_1\dotsm\mbox{d}s_4\\
    &\leq\frac{384}{2\eps^2}\left(\int_{[0,\infty)^2}\ind_{\{s_1\leq s_2\}}\lvert f(s_1)f(s_2)\rvert e^{-2\frac{s_2-s_1}{\eps^2}(1-\cos\theta)}\mbox{d}s_1\mbox{d}s_2\right)^2\\
    &\leq3\left(\frac{4}{1-\cos\theta}\int_0^\infty f^2(s)\mbox{d}s\right)^2.
\end{align*}

Then the family $\{Y_\eps^f\tilde{Y}_\eps^g\}_{\eps>0}$ is uniformly
integrable and consequently
\[
\esp\left[\lim_{\eps\rightarrow
0}Y_\eps^f(t)\tilde{Y}_\eps^g(s)\right]=\lim_{\eps\rightarrow
0}\esp[Y_\eps^f(t)\tilde{Y}_\eps^g(s)].
\]
Since $Y^f$ and $\tilde{Y}^g$ are centered Gaussian processes, in
order to prove their independence it suffices to show that the last
limit converges to zero as $\eps$ tends to zero. To deal with this
limit, we observe that
\begin{align*}
\esp[Y_\eps^f\tilde{Y}_\eps^g]&=\frac4{\eps^2}\int_0^\infty\int_0^\infty f(s)g(r)\esp\left[\cos(\theta N_{\frac{2s}{\eps^2}})\sin(\theta N_{\frac{2r}{\eps^2}})\right]\mathrm{d}s\mathrm{d}r\\
    &=\frac4{\eps^2}\int_0^\infty\int_0^\infty f(s)g(r)\ind_{\{s\leq r\}}\esp\left[\cos(\theta N_{\frac{2s}{\eps^2}})\sin(\theta N_{\frac{2r}{\eps^2}})\right]\mathrm{d}s\mathrm{d}r\\
    &\quad+\frac4{\eps^2}\int_0^\infty\int_0^\infty f(s)g(r)\ind_{\{r\leq s\}}\esp\left[\cos(\theta N_{\frac{2s}{\eps^2}})\sin(\theta N_{\frac{2r}{\eps^2}})\right]\mathrm{d}s\mathrm{d}r\\
    &=I_1^\eps+I_2^\eps.
\end{align*}
Applying the formula $2\sin a\cos
b=\sin(a+b)+\sin(a-b)=\sin(a+b)-\sin(b-a)$ we have
\begin{align*}
I_1^\eps&=\frac2{\eps^2}\int_0^\infty\int_0^\infty f(s)g(r)\ind_{\{s\leq r\}}\esp\left[\sin(\theta(N_{\frac{2s}{\eps^2}}+N_{\frac{2r}{\eps^2}}))+\sin(\theta(N_{\frac{2r}{\eps^2}}-N_{\frac{2s}{\eps^2}}))\right]\mathrm{d}s\mathrm{d}r\\
    &=I_{1,1}^\eps+I_{1,2}^\eps,
\end{align*}
\begin{align*}
I_2^\eps&=\frac2{\eps^2}\int_0^\infty\int_0^\infty f(s)g(r)\ind_{\{r\leq s\}}\esp\left[\sin(\theta(N_{\frac{2s}{\eps^2}}+N_{\frac{2r}{\eps^2}}))-\sin(\theta(N_{\frac{2s}{\eps^2}}-N_{\frac{2r}{\eps^2}}))\right]\mathrm{d}s\mathrm{d}r\\
    &=I_{2,1}^\eps-I_{2,2}^\eps.
\end{align*}
We proceed to show that $I_{1,1}^\eps$ and $I_{2,1}^\eps$ converges
to zero as $\eps$ tends to zero and that $I_{1,2}^\eps$ and
$I_{2,2}^\eps$ have the same (finite) limit, thus obtaining the
stated result. We note that
\begin{align*}
I_{1,1}^\eps&=\mathrm{Im}\left(\frac2{\eps^2}\int_0^\infty\int_0^\infty f(s)g(r)\ind_{\{s\leq r\}}\esp\left[e^{i\theta(N_{\frac{2r}{\eps^2}}-N_{\frac{2s}{\eps^2}})}e^{2i\theta N_{\frac{2s}{\eps^2}}}\right]\mathrm{d}s\mathrm{d}r\right)\\
    &=\mathrm{Im}(A^\eps).
\end{align*}
To find the limit of $I_{1,1}^\eps$ we see that $\lVert
A^\eps\rVert$ converges to zero as $\eps$ tends to zero.
\begin{align*}
\lVert A^\eps\rVert&\leq\frac2{\eps^2}\int_0^\infty\int_0^\infty\lvert f(s)g(r)\rvert\ind_{\{s\leq r\}}e^{-\frac2{\eps^2}(r-s)(1-\cos\theta)}e^{-\frac2{\eps^2}s(1-\cos2\theta)}\mathrm{d}s\mathrm{d}r\\
        &\leq\frac1{\eps^2}\int_0^\infty\int_0^\infty\left(f^2(s)+g^2(r)\right)\ind_{\{s\leq r\}}e^{-\frac{2r}{\eps^2}(1-\cos\theta)}e^{\frac{2s}{\eps^2}(\cos2\theta-\cos\theta)}\mathrm{d}s\mathrm{d}r\\
        &=\frac1{\eps^2}\int_0^\infty f^2(s)e^{\frac{2s}{\eps^2}(\cos2\theta-\cos\theta)}\int_s^\infty e^{-\frac{2r}{\eps^2}(1-\cos\theta)}\mathrm{d}r\mathrm{d}s\\
        &\quad+\frac1{\eps^2}\int_0^\infty g^2(r)e^{-\frac{2r}{\eps^2}(1-\cos\theta)}\int_0^r e^{\frac{2s}{\eps^2}(\cos2\theta-\cos\theta)}\mathrm{d}s\mathrm{d}r\\
        &=A_1^\eps+A_2^\eps.
\end{align*}
When $\cos\theta=\cos2\theta$ it is easy to check the convergence to
zero. Otherwise, we integrate obtaining
\[
A_1^\eps=\frac1{2(1-\cos\theta)}\int_0^\infty
f^2(s)e^{-\frac{2s}{\eps^2}(1-\cos2\theta)}\mathrm{d}s,
\]
\begin{align*}
A_2^\eps&=\frac1{2(\cos2\theta-\cos\theta)}\int_0^\infty g^2(r)e^{-\frac{2r}{\eps^2}(1-\cos\theta)}\left(e^{\frac{2r}{\eps^2}(\cos2\theta-\cos\theta)}-1\right)\mathrm{d}r\\
        &=\frac1{2(\cos2\theta-\cos\theta)}\int_0^\infty g^2(r)\left(e^{-\frac{2r}{\eps^2}(1-\cos2\theta)}-e^{-\frac{2r}{\eps^2}(1-\cos\theta)}\right)\mathrm{d}r
\end{align*}
which concludes, as the convergence to zero is easily seen by
dominated convergence. In the same manner we can see that
$I_{2,1}^\eps$ converges to zero.

With respect to the term $I_{1,2}^\eps$ we observe that
\[
I_{1,2}^\eps=\mathrm{Im}\left(\frac2{\eps^2}\int_0^\infty\int_0^\infty
f(s)g(r)\ind_{\{s\leq
r\}}e^{-\frac2{\eps^2}(r-s)(1-e^{i\theta})}\mathrm{d}s\mathrm{d}r\right).
\]
Since
$\frac2{\eps^2}(1-e^{i\theta})e^{-\frac2{\eps^2}(r-s)(1-e^{i\theta})}$
is an approximation of the identity, we have that $I_{1,2}^\eps$
converges, as $\eps$ tends to zero, to
$\mathrm{Im}\left(\frac1{1-e^{i\theta}}\int_0^\infty
f(s)g(s)\mathrm{d}s\right)<\infty.$ Clearly the same result is
obtained for $I_{2,2}^\eps$. This finishes the proof.
\end{proof}

\begin{remark}
We can use this result to approximate two independent processes of
many kinds, such as processes with a Gousart kernel (see for
instance \cite{Delgado-2000}), the Holmgren-Riemann-Liouville
fractional integral (\cite{Delgado-2000}), fractional Brownian
motion and sub-fractional Brownian motion. In the next section we
study weak convergence, in the space $\mathcal C([0,T])$, towards
these two last processes.
\end{remark}


\section{Weak approximation of fractional and sub-fractional Brownian motion}

In this section we apply the main theorem to prove weak convergence
results towards fractional Brownian motion and sub-fBm. Let us begin
with fBm and later on we will study the convergence for sub-fBm.

\bigskip
\subsection{Weak approximation of fractional Brownian motion} \noindent\bigskip

We are going to prove a result of weak convergence in
$\mathcal{C}([0,T])$ towards fBm, applying Theorem \ref{t:conv}. In
order to do so, we use the following representation of the fBm as
the integral of a deterministic kernel with respect to standard
Brownian motion (see for instance \cite{decreuse})
\begin{equation}\label{e:fBm-Bm}
B_t^H=\int_0^t\tilde{K}^H(t,s)\,\mbox{d}W_s,
\end{equation}
where $H\in(0,2)$, $\tilde{K}^H(t,s)$ is defined on the set
$\{0<s<t\}$ and is given by
\begin{equation}\label{e:nucli-fBm}
\tilde{K}^H(t,s)=d^H(t-s)^{\frac{H-1}{2}}+d^H\left(\frac{1-H}{2}\right)\int_s^t(u-s)^{\frac
{H-3}2}\left(1-\left(\frac{s}{u}\right)^{\frac{1-H}{2}}\right)\mbox{d}u,
\end{equation}
where the normalizing constant $d^H$ is
\[
d^H=\left(\frac{H\Gamma (\frac{3-H}{2})}{\Gamma(\frac {H+1}2)\Gamma
(2-H)}\right)^\frac{1}{2}.
\]
Since in this section the domain of fBm is restricted to the
interval $t\in[0,T]$, we can rewrite the integral representation as
\[
B_t^H=\int_0^t\tilde{K}^H(t,s)\,\mbox{d}W_s=\int_0^TK^H(t,s)\,\mbox{d}W_s,
\]
where $K^H(t,s)=\tilde{K}^H(t,s)\ind_{[0,t]}(s)$.

Applying this representation, since
$K^H(t,\cdot)\in\mathrm{L}^2(\R^+)$, the following result is a
corollary of Theorem \ref{t:conv}

\begin{corollary}\label{t:conv-frac}
Let $K^H(t,s)=\tilde{K}^H(t,s)\ind_{[0,t]}(s)$, where
$\tilde{K}^H(t,s)$ is defined by (\ref{e:nucli-fBm}), let
$\{N_s,s\geq0\}$ be a standard Poisson process and let
$\theta\in(0,\pi)\cup(\pi,2\pi)$. Then the processes
\begin{equation}\label{e:cos_frac}
B^H_{\eps}=\left\{\frac{2}{\eps}\int_0^TK^H(t,s)\cos\left(\theta
N_{\frac{2s}{\eps^2}}\right)\mbox{d}s,\quad t\in[0,T]\right\}
\end{equation}
and
\begin{equation}\label{e:sin_frac}
\tilde{B}^H_{\eps}=\left\{\frac{2}{\eps}\int_0^TK^H(t,s)\sin\left(\theta
N_{\frac{2s}{\eps^2}}\right)\mbox{d}s,\quad t\in[0,T]\right\}
\end{equation}
converge in law, in the sense of the finite dimensional
distributions, towards two independent fractional Brownian motions.
\end{corollary}

We now proceed to prove the continuity and the tightness of the
families of processes defined by (\ref{e:cos_frac}) and
(\ref{e:sin_frac}), and consequently, proving the weak convergence
in the space $\mathcal C([0,T])$.

\begin{theorem}\label{t:conv-frac2}
Under the hypothesis of Corollary \ref{t:conv-frac}, if moreover one
of the following conditions is satisfied:
\begin{enumerate}
\item $H\in(\frac{1}{2},2)$,\label{i:cas1}
\item $H\in(0,\frac{1}{2}]$ and $\theta$ satisfies $\cos((2i+1)\theta)\neq1$ for all $i\in\N$ such that $i\leq\frac{1}{2}\left[\frac{1}{H}\right]$, \label{i:cas2}
\end{enumerate}
then the processes $B^H_\eps$ and $\tilde{B}^H_\eps$ converge in law
in $\mathcal{C}([0,T])$ towards two independent fractional Brownian
motions.
\end{theorem}
\begin{proof}
We first observe that the processes $B^H_{\eps}$ and
$\tilde{B}^H_{\eps}$ are continuous. In fact, $B^H_{\eps}$ and
$\tilde{B}^H_{\eps}$ are continuous for all $H\in(0,2)$ and
absolutely continuous if $H\in(1,2)$, since it can be proved that
(see Lemma 2.1 in \cite{Bardina-2005})
\[
\lvert B^H_{\eps}(t)- B^H_{\eps}(s)\rvert\leq
C_H(t-s)^{\left(\frac{H+1}2\right)\wedge1}
\]
and
\[
\lvert \tilde{B}^H_{\eps}(t)- \tilde{B}^H_{\eps}(s)\rvert\leq
C_H(t-s)^{\left(\frac{H+1}2\right)\wedge1}.
\]
It only remains to prove the tightness of the families of processes
defined by (\ref{e:cos_frac}) and (\ref{e:sin_frac}). Since
$B_{\varepsilon}^H(0)=0$, using Billingsley's criterion (see for
instance \cite{Billingsley-1968}) it is enough to check that for
some $m>0$ and $\alpha>1$
$$\esp[|B_{\varepsilon}^H(t)-B_{\varepsilon}^H(s)|^m]\leq C(F(t)-F(s))^{\alpha},$$
where $F$ is a nondecreasing continuous function.

On the other hand, it is known that
\[
\int_0^T\left(K^H(t,r)-K^H(s,r)\right)^2\mathrm{d}r=\esp\left[(B^H_t-B^H_s)^2\right]=(t-s)^H,
\]
and then it is sufficient to show that
\begin{equation}\label{e:m}
\esp\left[(y_{\eps}^f)^m\right]\leq
C_m\left(\int_0^Tf^2(r)\,\mbox{d}r\right)^{\frac{m}{2}}\!,\quad
\esp\left[(\tilde{y}_{\eps}^f)^m\right]\leq
C_m\left(\int_0^Tf^2(r)\,\mbox{d}r\right)^{\frac{m}{2}}
\end{equation}
holds for some $m$ satisfying the condition $Hm/2>1$, where
$f(r):=K^H(t,r)-K^H(s,r)$,
$y_{\eps}^f=\frac{2}{\eps}\int_0^Tf(r)\cos(\theta
N_{\frac{2r}{\eps^2}})\mbox{d}r$ and
$\tilde{y}_{\eps}^f=\frac{2}{\eps}\int_0^Tf(r)\sin(\theta
N_{\frac{2r}{\eps^2}})\mbox{d}r$.

Then, in the case (\ref{i:cas1}), it is sufficient to prove
(\ref{e:m}) for $m=4$, which can be seen proving that
$\esp[(z_\eps^f\bar{z}_\eps^f)^2]\leq C\lVert f\rVert_2^4$, where
$\lVert\cdot\rVert_2$ is the $L^2[0,T]$ norm and
$z_\eps^f=y_\eps^f+i\tilde{y}_\eps^f$. If we extend $f$ to $\R^+$
for zeros, i.e., if we consider $F(r):=f(r)\ind_{[0,T]}(r)$, we have
proved in Theorem \ref{t:conv} that
\[
\esp[(Z_\eps^F\bar{Z}_\eps^F)^2]\leq3\left(\frac{4}{1-\cos\theta}\int_0^\infty
F^2(s)\mbox{d}s\right)^2.
\]
Then,
\begin{multline*}
\esp[(z_\eps^f\bar{z}_\eps^f)^2]=\esp[(Z_\eps^F\bar{Z}_\eps^F)^2]\\
\leq3\left(\frac{4}{1-\cos\theta}\int_0^\infty
F^2(s)\mbox{d}s\right)^2=3\left(\frac{4}{1-\cos\theta}\int_0^T
f^2(s)\mbox{d}s\right)^2.
\end{multline*}

To prove the result under the hypothesis (\ref{i:cas2}) we can show
that (\ref{e:m}) is satisfied for some even $m$ such that
$\frac{Hm}{2}>1$. If we proceed in the same way as in case
(\ref{i:cas1}) we obtain an expression that depends on
$1-\cos((2i+1)\theta)$ for all
$i=0,1,\dotsc,\left[\frac{1}{2H}\right]$ and the constant $C_m$
depends on
$\max_{i=0,1,\dotsc,\left[\frac{1}{2H}\right]}\frac{1}{1-\cos((2i+1)\theta)}$.
\end{proof}

\bigskip
\subsection{Convergence towards sub-fractional Brownian motion}
\noindent\bigskip

In order to obtain the convergence to sub-fractional Brownian
motion, we will apply a decomposition result due to Ruiz de Ch\'avez
and Tudor in \cite{Tudor3}. In this paper, they use the process
$X^H$ introduced by Lei and Nualart in \cite{Lei-2008} and defined
in (\ref{e:david}) by the equation
$$
X_t^H=\int_0^{\infty}(1-e^{-r t})r^{-\frac{1+H}{2}}\,\mathrm{d}W_r,
$$
where $W$ is a standard Brownian motion. It can be proved (see
\cite{Lei-2008} or \cite{Tudor3}) that its covariance function is
\begin{equation}\label{e:david_cov}
\mathrm{Cov}(X_t^H,X_s^H)=\left\{
\begin{array}{ll}
\tfrac{\Gamma(1-H)}{H}\left[t^H+s^H-(t+s)^H\right] & \mbox{ if $H\in (0,1)$,} \\
\tfrac{\Gamma(2-H)}{H(H-1)}\left[(t+s)^H-t^H-s^H\right] & \mbox{ if
$H\in (1,2)$},
\end{array}
\right.
\end{equation}
and that $X^H$ has a version with absolutely continuous trajectories
on $[0,\infty)$.

The decomposition result can be stated and proved as follows:

\begin{theorem}[Decomposition of sub-fBm]\label{t:descomposicio}
Let $B^H$ be a fBm, $S^H$ a sub-fBm and $W=\{W_t,t\geq0\}$ a
standard Brownian motion. Let $X^H$ be the process given by
(\ref{e:david}). If for $H\in(0,1)$ we suppose that $B^H$ and $W$
are independents, then the processes
$\{Y_t^H=C_1X_t^H+B_t^H,t\geq0\}$ and $\{S_t^H,t\geq0\}$ have the
same law, where $C_1=\sqrt{\frac{H}{2\Gamma(1-H)}}$. If for
$H\in(1,2)$ we suppose that $S^H$ and $W$ are independents, then the
processes $\{Y_t^H=C_2X_t^H+S_t^H,t\geq0\}$ and $\{B_t^H,t\geq0\}$
have the same law, where $C_2=\sqrt{\frac{H(H-1)}{2\Gamma(2-H)}}$.
\end{theorem}
\begin{proof}
It is clear that the process $Y^H$ is centered and Gaussian in both
cases. For $H\in(0,1)$, from (\ref{e:frac_cov}), (\ref{e:david_cov})
and using the independence of $X^H$ and $B^H$ we have
\begin{align*}
\mbox{Cov}(Y_t^H,Y_s^H)&=C_1^2\mbox{Cov}[X_t^H,X_s^H]+\mbox{Cov}[B _t^H,B _s^H]\\
                       &=s^{H}+t^{H}-\frac{1}{2}\left[(s+t)^{H}+|s-t|^{H}\right],
\end{align*}
which completes the proof in this case, and for $H\in(1,2)$, from
(\ref{e:subfrac_cov}), (\ref{e:david_cov}) and using the
independence of $X^H$ and $S^H$ we have
\begin{align*}
\mbox{Cov}(Y_t^H,Y_s^H)&=C_2^2\mbox{Cov}[X_t^H,X_s^H]+\mbox{Cov}[S _t^H,S _s^H] \\
                       &=\frac{1}{2}\left(s^{H}+t^{H}-|s-t|^{H}\right),
\end{align*}
which completes the proof.
\end{proof}

In order to apply the main theorem to prove weak convergence to
sub-fBm, we have to prove weak convergence to fBm and the process
$X^H$ introduced by Lei and Nualart. Then, applying the
decomposition theorem and the independence of the limit laws, we can
state the weak convergence to sub-fBm for $H\in(0,1)$.

So, it just remains to prove for the process $X^H$ defined by
(\ref{e:david}) the same results we have obtained for fBm.

\begin{corollary}\label{t:conv-david}
Let $X^H$be the process defined by (\ref{e:david}), let
$\{N_s,s\geq0\}$ be a standard Poisson process and let
$\theta\in(0,\pi)\cup(\pi,2\pi)$. Then the processes
\begin{equation}\label{e:cos_david}
X^H_{\eps}=\left\{\frac{2}{\eps}\int_0^\infty(1-e^{-st})s^{-\frac{1+H}{2}}\cos\left(\theta
N_{\frac{2s}{\eps^2}}\right)\mbox{d}s,\quad t\in[0,T]\right\}
\end{equation}
and
\begin{equation}\label{e:sin_david}
\tilde{X}^H_{\eps}=\left\{\frac{2}{\eps}\int_0^\infty(1-e^{-st})s^{-\frac{1+H}{2}}\sin\left(\theta
N_{\frac{2s}{\eps^2}}\right)\mbox{d}s,\quad t\in[0,T]\right\}
\end{equation}
converge in law, in the sense of the finite dimensional
distributions, towards two independent processes with the same law
that $X^H$.
\end{corollary}

\begin{theorem}\label{t:conv-david2}
Under the hypothesis of Corollary \ref{t:conv-david} the processes
defined by (\ref{e:cos_david}) and (\ref{e:sin_david}) converge in
law, in $\mathcal{C}([0,T])$, towards two independent processes with
the same law that the process defined by (\ref{e:david}).
\end{theorem}
\begin{proof}
We first need to show that the processes $X^H_{\eps}$ and
$\tilde{X}^H_{\eps}$ are continuous. In fact, they are absolutely
continuous. Let us consider for all $r>0$ the process
\[
Y_r=\frac2\eps\int_0^\infty s^{\frac{1-H}2}e^{-sr}\cos\left(\theta
N_{\frac{2s}{\eps^2}}\right)\mbox{d}s.
\]
This integral exists because, using inequality (\ref{e:m2}), we have
\[
\esp[Y_r^2]\leq C\left(\int_0^\infty s^{1-H}e^{-2sr}\mathrm
ds\right)=Cr^{H-2}\Gamma(2-H).
\]
On the other hand,
\[
\esp\left[\int_0^t\lvert
Y_r\rvert\mathrm{d}r\right]\leq\int_0^t(\esp[Y_r^2])^{\frac12}\mathrm{d}r\leq
C\int_0^tr^{\frac{H-2}2}\mathrm{d}r<\infty
\]
since $H\in(0,2)$.

Let us now observe that $X^H_\eps=\int_0^t Y_r\mathrm{d}r$. Indeed,
applying Fubini's theorem,
\begin{align*}
\int_0^tY_r\mathrm{d}r&=\frac2\eps\int_0^\infty s^{\frac{1-H}2}\left(\int_0^t e^{-sr}\mathrm{d}r\right)\cos\left(\theta N_{\frac{2s}{\eps^2}}\right)\mbox{d}s\\
        &=\frac2\eps\int_0^\infty s^{-{\frac{1+H}2}}(1-e^{-st})\cos\left(\theta N_{\frac{2s}{\eps^2}}\right)\mbox{d}s\\
        &=X^H_\eps.
\end{align*}
The same proof shows that the process $\tilde X^H_\eps$ is
continuous.

Next, we prove the convergence only for (\ref{e:cos_david}). For
(\ref{e:sin_david}) the result is proved similarly.

It suffices to prove the tightness of the family
$\{X_\eps^H\}_\eps$. Since $X_\eps^H(0)=0$, using Billingsley's
criterion we only need to prove that
\[
\esp\left[\lvert X_\eps^H(t)-X_\eps^H(s)\rvert^4\right]\leq\lvert
F(t)-F(s)\rvert^2
\]
where $F$ is a continuous, non-decreasing function. We observe that
\[
\esp\left[\lvert
X_\eps^H(t)-X_\eps^H(s)\rvert^4\right]=\esp\left[\frac{2}{\eps}\int_0^\infty\left(\Phi^H(t,r)-\Phi^H(s,r)\right)\cos(\theta
N_{\frac{2r}{\eps^2}})\mathrm{d}r\right]^4
\]
where
$\Phi^H(t,r)=(1-e^{-rt})r^{-\frac{1+H}{2}}\in\mathrm{L}^2(\R^+)$.

Since $\Phi^H\in\mathrm{L}^2(\R^+)$, applying the bound
(\ref{e:m4}), which is proved in Theorem \ref{t:conv}, we obtain

\begin{align*}
\esp\left[\lvert X_\eps^H(t)-X_\eps^H(s)\rvert^4\right]&\leq C\left(\int_0^\infty\left(\Phi^H(t,r)-\Phi^H(s,r)\right)^2\mathrm{d}r\right)^2\\
        &=C\bigg(\int_0^\infty\Big((1-e^{-rt})^2r^{-(1+H)}+(1-e^{-rs})^2r^{-(1+H)}\\
        &\quad-2(1-e^{-rt})(1-e^{-rs})r^{-(1+H)}\Big)\mathrm{d}r\bigg)^2.
\end{align*}
Using (\ref{e:david_cov}) and assuming $s<t$ we obtain for
$H\in(0,1)$

\begin{align*}
\esp\left[\lvert X_\eps^H(t)-X_\eps^H(s)\rvert^4\right]&\leq C\left(2(t+s)^H-(2t)^H-(2s)^H\right)^2\\
        &\leq C\left((2t)^H-(2s)^H\right)^2,
\end{align*}
since $s+t<2t$. In the same way, if $H\in(1,2)$,
\begin{align*}
\esp\left[\lvert X_\eps^H(t)-X_\eps^H(s)\rvert^4\right]&\leq C\left((2t)^H+(2s)^H-2(t+s)^H\right)^2\\
        &\leq C\left((2t)^H-(2s)^H\right)^2,
\end{align*}
since $s+t>2s$. In both cases we have proved the result with
$F(x)=(2x)^H$.
\end{proof}


Finally, we obtain the result of weak convergence to sub-fractional
Brownian motion, as a direct conclusion of the previous results.

\begin{theorem}\label{t:conv-sub}
Let $H\in(0,1)$, let $\{X_\eps^H(t),t\in[0,T]\}$ be the family of
processes defined by (\ref{e:cos_david}), let
$\{\tilde{B}^H_\eps(t),t\in[0,T]\}$ be the family of processes
defined by (\ref{e:sin_frac}) and
$C_1=\sqrt{\frac{H}{2\Gamma(1-H)}}$. Let us assume
$\theta\in(0,\pi)\cup(\pi,2\pi)$ and, for $H\in(0,\frac{1}{2}]$,
that $\theta$ is such that $\cos((2i+1)\theta)\neq1$ for all
$i\in\N$ such that $i\leq\frac{1}{2}\left[\frac{1}{H}\right]$. Then,
$\{Y^H_\eps(t)=C_1X^H_\eps(t)+B^H_\eps(t), t\in[0,T]\}$ weakly
converges in $\mathcal C([0,T])$ to a sub-fractional Brownian
motion.
\end{theorem}
\begin{proof}
Applying Theorems \textbf{\ref{t:conv-frac2}} and
\textbf{\ref{t:conv-david2}} we know that, respectively, the
processes  $\tilde{B}^H_\eps$ and $X_\eps^H$ converge in law in
$\mathcal C([0,T])$ towards a fBm and the process defined by
(\ref{e:david}). Moreover, applying Theorem \textbf{\ref{t:conv}},
we know that the limit laws are independent. Hence, we are under the
hypothesis of Theorem \textbf{\ref{t:descomposicio}}, which proves
the stated result.
\end{proof}

\begin{remark}
Obviously we can also obtain the same result using the families of
processes defined by (\ref{e:sin_david}) and (\ref{e:cos_frac}).
\end{remark}

\bibliographystyle{alpha}
\bibliography{WeakConvtoGP}
\end{document}